\documentclass[12pt,a4paper]{amsart}
\RequirePackage{iftex}

\ifdefined\pdfoutput
  \pdfoutput=1
\fi

\ifPDFTeX
  \PassOptionsToPackage{pdftex}{graphicx}
  \PassOptionsToPackage{pdftex}{xcolor}
\else
  \PassOptionsToPackage{dvipdfmx}{graphicx}
  \PassOptionsToPackage{dvipdfmx}{xcolor}
\fi

\usepackage{graphicx}
\usepackage{xcolor}
\usepackage{tikz}

\usepackage{amsmath,amsthm,mathtools, amsfonts,amssymb}

\usepackage[utf8]{inputenc}
\usepackage{graphicx} 
\usepackage{color}
\usepackage{xcolor}
\usepackage{tikz}
\usepackage{tikz-cd}
\usepackage[all]{xy}
\usetikzlibrary{cd,positioning}
\usepackage{hyperref}
\usepackage{extarrows}
\usepackage{ytableau}

\usepackage[mathscr]{euscript}

\usepackage[colorinlistoftodos]{todonotes}
\usepackage{amsmath}
\allowdisplaybreaks[4] 
\numberwithin{equation}{section}

\usepackage{tikz}

\date{}

\newcommand{\Inv}{\mathrm{Inv}}
\newcommand{\Des}{\mathrm{Des}}

\newcommand{\Gr}{\mathrm{Gr}}
\newcommand{\loc}{\mathrm{loc}}
\newcommand{\OGr}{\ell}
\newcommand{\F}{\mathbb{F}}
\newcommand{\ePet}{\hat{\mathbb{L}}_G}
\newcommand{\minrep}[1]{\lfloor #1\rfloor_P}
\newcommand{\Pet}{\mathbb{L}_G}

\newcommand{\Z}{\mathbb{Z}}

\newcommand{\bO}{\mathbb{O}}

\newcommand{\piv}{\varpi^\lor}

\newcommand{\CO}{\mathscr{O}_{G/B}}

\newcommand{\COP}{\mathscr{O}_{G/P}}

\newcommand{\af}{\mathrm{af}}

 \newcommand{\miw}[1]{v_{#1}}
\newcommand{\ad}{\mathrm{ad}}

\newcommand{\Wexgr}{\hat{W}_\af^0}

\newcommand{\pair}[2]{\langle #1, #2 \rangle}
\DeclareMathOperator{\inv}{Inv}

\newtheorem{thm}{Theorem}[section]
\newtheorem{lem}[thm]{Lemma}
\newtheorem{prop}[thm]{Proposition}
\newtheorem{cor}[thm]{Corollary}

\newtheorem{example}[thm]{Example}

\newtheorem{defn}[thm]{Definition}
\theoremstyle{remark}
\newtheorem{remark}[thm]{Remark}
\title[Seidel product formula in equivariant quantum K-theory]{Seidel product formula in equivariant quantum $K$-theory of flag varieties}
\author{Takeshi Ikeda, Takafumi Kouno, Satoshi Naito}

\begin{document}
\begin{abstract}
We prove a Seidel product formula
for the torus-equivariant quantum $K$-theory of a generalized flag variety $G/P.$
This is a natural generalization of 
the corresponding results by Buch, Chaput, and Perrin for the cominuscule flag varieties.
Our proof is based on the $K$-theoretic Peterson isomorphism, due to Kato. We also use a version of the $K$-theoretic nil-Hecke algebra associated with the extended affine Weyl group, which was studied by Ikeda, Shimozono, and Yamaguchi. \end{abstract}
\maketitle

\section{Introduction}

Let \(G\) be a simple, simply connected complex algebraic group. Fix a
maximal torus \(T\) of \(G\) and a Borel subgroup \(B\) of \(G\)
containing \(T\).
We study the $T$-equivariant quantum $K$-theory\footnote{Throughout the paper, $QK_T(G/B)$ denotes the polynomial version of equivariant quantum $K$-theory, after Anderson--Chen--Tseng \cite{ACT}.} $QK_T(G/B)$
of the flag
variety $G/B$.
The ring $QK_T(G/B)$ admits the Schubert basis $\{\CO^w\}$ indexed by the
Weyl group $W$.

Seidel's construction \cite{S} has played an important role in quantum
cohomology and quantum $K$-theory.
In quantum $K$-theory, Buch, Chaput, and Perrin \cite{BCP} proved a Seidel
product formula for cominuscule flag varieties.
More recently, Li, Liu, Song, and Yang \cite{LLSY} studied the Seidel
representation for Grassmannians by using curve neighbourhood.
On the other hand, Chaput, Manivel, and Perrin  \cite{CMP,CP} established the
corresponding formula in equivariant quantum cohomology for arbitrary
$G/P$.

The main purpose of this paper is to prove a Seidel product formula in
$T$-equivariant quantum $K$-theory for general flag varieties $G/B$.
Our proof is a direct application of Kato's $K$-theoretic Peterson
isomorphism \cite{Kato} together with the extended Peterson algebra of
Ikeda--Shimozono--Yamaguchi \cite{ISY}.
Using Kato's pushforward formula \cite{Kato2}, we further obtain the corresponding
Seidel product formula for $QK_T(G/P)$ for any parabolic subgroup
$P\subset G$.


\subsection{Seidel classes}
Let $I$ denote the index set of the Dynkin nodes of $G$.
A Dynkin node  $i\in I$ of $G$ is called \emph{special} (or \emph{cominuscule}) 
if 
\begin{equation}
\langle \alpha,\piv_i\rangle\in \{0,1\}\quad \text{for all positive roots $\alpha$},
\label{eq:cominuscule}
\end{equation}
where $\piv_i$ is the fundamental coweight corresponding to $i\in I$ (see \S \ref{ssec:Sigma}, \S \ref{ssec:Seidel} for other aspects of special nodes). 
Let $W$ be the Weyl group generated by standard generators $s_i\;(i\in I)$ of
simple reflections. 
Let $P_i$ be the standard maximal parabolic subgroup associated with $I\setminus \{i\}.$ We consider  
an element $\miw{i}$ in Weyl group $W$ of $G$ defined by
\begin{equation*}
w_\circ= \miw{i}\,w_{P_i},
\end{equation*} where
$w_\circ$ and $w_{P_i}$ are the longest elements of $W$ and  
$W_{P_i}=\langle s_j\mid j\in I\setminus \{i\}\rangle,$ respectively. 
We call $\miw{i}$ the \emph{Seidel element} corresponding to a special node $i$.

The class $\CO^{\miw{i}}$, for special $i$, is 
called a \emph{Seidel class} since it is related to 
the so-called Seidel representation of the fundemental group $\pi_1(G^\ad)$ of the group $G^\ad=G/Z(G)$ of adjoint type, where $Z(G)$ is the center of $G.$
The group $\pi_1(G^\ad)$ is isomorphic to the quotient group $\Sigma:=P^\lor/Q^\lor,$ where $P^\lor$ and $Q^\lor$ are the coweight lattice and the coroot lattce,  respectively.
\subsection{Seidel product formula for $QK_T(G/B)$}
There is a natural $W$-action on $QK_T(G/B)$, called the \emph{left action},
defined in the setting of the torus-equivariant quantum $K$-theory.
We refer to \cite[\S 8.3]{MNS} for details.
For \(u\in W\), we denote by \(u^L\) the corresponding ring automorphism of \(QK_T(G/B)\).
In particular, the action on the (opposite) Schubert basis is given by
\eqref{eq:W acts CO} below.

\begin{thm}\label{thm:main}
Let $i$ be a special (cominuscule) node of the Dynkin diagram of $G,$
and $\miw{i}$ the Seidel element corresponding to $i$. Then,  for any $w\in W,$ we have 
    $$
\CO^{\miw{i}}\cdot \miw{i}^L{\CO^w}=Q^{\piv_i-w^{-1}(\piv_i)}\CO^{\miw{i} w}
$$
 in torus-equivatiant quantum $K$-theory $QK_T(G/B)$. 
\end{thm}

This generalizes the cominuscule case  treated by Buch--Chaput--Perrin \cite{BCP}
to the full flag variety $G/B$.

\subsection{Parabolic case}
Let $P$ be a standard parabolic subgroup of $G$. 
Let $I_P$ be the subset of $I$ corresponding to $P.$
Let $W_P=\langle s_j\mid  j\in I_P\rangle$
be the subgroup corresponding to $P$, and 
let $W^P$ denote the set of minimal-length coset representatives of $W/W_P$. For each $w\in W^P$ there is the Schubert class $\COP^w$ in $QK_T(G/P).$
For $w\in W$, let $\minrep{w}$ denote the unique element of $W^P$ such that $wW_P=\minrep{w}W_P.$
 For $\beta \in Q^\lor$, 
define 
\begin{equation*}
\minrep{\beta}
=
\beta - \sum_{j \in I_P}
\langle \varpi_j^\lor, \beta \rangle \, \alpha_j^\lor .
\end{equation*}
\begin{cor}\label{cor:P}
In $QK_T(G/P)$, we have
\begin{equation*}
\COP^{\minrep{\miw{i}}}\cdot \miw{i}^L{\COP^w}=Q^{\minrep{\piv_i-w^{-1}(\piv_i)}}\COP^{\minrep{\miw{i} w}}
\quad\text{for $w\in W^P.$}
\end{equation*}
 \end{cor}
This is the exact $K$-theoretic analogue of the Seidel product formula of
Chaput--Perrin \cite{CP}.
This corollary follows immediately from Theorem~\ref{thm:main} by applying Kato's
pushforward homomorphism
\[\pi_*:
QK_T(G/B)\longrightarrow QK_T(G/P)
\]
in equivariant quantum $K$-theory \cite{Kato2}.
In particular, in the cominuscule case $P=P_i$, this recovers the
non-equivariant formula of Buch--Chaput--Perrin \cite{BCP}.


\subsection{$K$-theoretic Peterson isomorphism}

Our proof relies on the ``quantum equals affine'' philosophy, which goes
back to Peterson \cite{Pet}.

Let $\Gr_G$ be the affine Grassmannian of $G$.
The $T$-equivariant $K$-homology $K_*^T(\Gr_G)$ is equipped with the
Pontryagin product (see \cite{LSS,Kumar:conj}).
A central ingredient is the $K$-theoretic Peterson isomorphism,
conjectured in \cite{LLMS} and established by Kato \cite{Kato}.
It yields an isomorphism of $R(T)$-algebras (see Theorem~\ref{thm:Kato})
\[
\Phi:
K_*^T(\Gr_G)_{\loc}\xrightarrow{\ \sim\ } QK_T(G/B)_{\loc}.
\]
Here the subscript $\loc$ indicates suitable localizations on both sides.

{\subsection{Extended $K$-theoretic Peterson algebra}

We also use an extended version of
the $K$-theoretic nil-Hecke algebra developed by
Ikeda--Shimozono--Yamaguchi \cite{ISY}.

Let $\hat{W}_{\af}=\Sigma\ltimes W_{\af}$ be the extended affine Weyl
group. In \cite{ISY}, an $R(T^{\ad})$-algebra $\hat{\mathbb{K}}_{\af}$ was constructed.
Inside $\hat{\mathbb{K}}_{\af}$, the authors define a commutative subalgebra
$\ePet$, called the extended $K$-theoretic Peterson algebra, which provides
an algebraic model for $K_*^{T^{\ad}}(\Gr_{G^{\ad}})$.
We refer to \cite{ISY} for precise definitions.
$\hat{\mathbb{K}}_{\af}$ acts on $\ePet$ naturally. 
This action incorporates the natural $\hat{W}_{\af}$-symmetry, in
particular the action of $\Sigma\simeq \pi_1(G^{\ad})$.
}

\subsection{Related works}
The 
Seidel product formula may be viewed as reflections of the hidden \emph{affine symmetry} of  Gromov-Witten invariants. Related symmetry phenomena in quantum Schubert calculus have been  investigated by several authors, including 
Chaput, Manivel, and Perrin \cite{CMP2010}, 
Agnihotri and Woodward \cite{AW}, 
Belkale \cite{B}, Postnikov \cite{P}.
Rooted in Seidel’s original construction, related connections with affine Schubert calculus were studied by Chow in (equivariant) quantum cohomology \cite{C} and by Chow and Leung in quantum 
$K$-theory \cite{CL}.



\subsection*{Organization}
The paper is organized as follows. In \S\ref{sec:ePet}, we review the extended $K$-Peterson algebra and its $K$-Peterson subalgebra, restricting ourselves to the minimal facts needed for the proof of the main theorem. In \S\ref{sec:proof}, we give the proof of our Seidel product formula, and its Corollary. In Appendix  \ref{sec:ePetA}, we present a review on the construction of the extended 
$K$-Peterson algebra.
Finally, in Appendix  \ref{sec:exa}, we illustrate the main results through several explicit examples.

\section{Extended K-theoretic Peterson algebra: Preliminaries}\label{sec:ePet}
We begin by reviewing the extended version of the 
$K$-theoretic Peterson algebra \cite{ISY,LSS}, restricting ourselves to the facts needed for the proof of the main result. Further details and additional constructions are deferred to Appendix \ref{sec:ePetA}.
Next, we recall the 
$K$-theoretic Peterson isomorphism,  due to Kato \cite{Kato}.  We also discuss some basic properties of the Seidel element 
$\miw{i}$ in \S \ref{sec:Seidel}. 

\subsection{Root systems}
Let $R$ be the root system of $(G,T)$. Let $G^\ad$ be the group of adjoint type associated with $R$. Let $I$ denote the Dynkin node sets. $P^\lor$ and $Q^\lor$ are the coweight lattice and coroot lattice,  respectively, with standard basis $\{\piv_i\mid i\in I\}$ and $\{\alpha_i^\lor\mid i\in I\}$
of the fundamental coweights, and the simple coroots, respectively.

\subsection{Extended affine Weyl group $\hat{W}_\af$}\label{ssec:Sigma}
Let $I_\af=I\cup\{0\}$ be the Dynkin node set of the untwisted affine type root system associated with the underlying root system $R.$
The extended affine Weyl group $\hat{W}_\af$ is generated by 
$\Sigma\cong P^\lor/Q^\lor$
and $W_\af=\langle s_i\mid i\in I_\af\rangle. $
 $\Sigma$ acts on the affine Dynkin node set $I_\af$ as automorphisms of the graph,  and for each special node $i$, there is a unique element in $\Sigma$ denoted by $\pi_i$, sending $0$ to $i$ 
(see 
\cite[\S 10.1]{LS:Acta}). 

Let $W_\af^0$ and $\hat{W}_\af^0$ be the set of affine Grassmannian elements for 
$W_\af$ and $\hat{W}_\af$,  respectively. We have $\hat{W}_\af^0=\Sigma \cdot W_\af^0.$
\subsection{Extended $K$-theoretic Peterson algebra}
\label{ssec:Extended Pet}
Let $\Pet\subset \mathbb{K}_\af$ and $\ePet\subset \hat{\mathbb{K}}_\af$ be the $K$-theoretic Peterson subalgebra and
its extended version.
We refer the construction of these algebras to \cite{ISY}, however, we give a brief review in Appendix  \ref{sec:ePetA} below, for the reader's convenience. 
$\Pet$ is a commutative $R(T)$-algebra, while
$\ePet$ is a commutative
$R(T^\ad)$-algebra. 
 $\Pet$ has a basis $\OGr_{x}\;(x\in {W}_\af^0)$ as an $R(T)$-module, while  
$\ePet$ has a basis $\OGr_{x}\;(x\in \hat{W}_\af^0)$ as an $R(T^\ad)$-module. 
There is an action of $\hat{\mathbb{K}}_\af$ on $\ePet$ called the \emph{star action} (see \cite[\S 2.4]{ISY}). In particular, 
the Weyl group $W$, which can be considered as a subset of $\hat{\mathbb{K}}_\af$ acts on $\ePet$ by the star action.


\begin{lem}[{\cite[Lemma 2.26, (2.61)]{ISY}}, cf. \cite{IIS}]
\label{lem:Sigma acts on L}
For $x\in \hat{W}_\af^0$ and  $\sigma=t_{\gamma_\sigma} u_\sigma\in \Sigma$,
with $\gamma_\sigma\in P^\lor$ and $u_\sigma\in W$, we have
\begin{equation}\label{eq:pi on ell}
\OGr_{\sigma x}=\OGr_\sigma( u_\sigma*\OGr_x),
\end{equation}
where $u_\sigma*\OGr_x$ denotes the the star action of $u_\sigma$ on $\OGr_x.$
In particular, for $\sigma,\sigma'\in \Sigma$, we have
\begin{equation}
\OGr_{\sigma\sigma'}=
\OGr_{\sigma}(u_{\sigma}*\OGr_{\sigma'}).
\label{eq: ell sigma sigma}
\end{equation}
\end{lem}

\begin{prop}[{\cite[Proposition 2.32]{ISY}}]\label{prop:t factor}
For anti-dominant $\gamma\in P^\lor$, 
and $x\in \hat{W}_\af^0$, we have
\begin{equation*}
\OGr_{xt_\gamma}=
\OGr_{x}\OGr_{t_\gamma}.
\end{equation*}
\end{prop}

For $i\in I,$
define $
\sigma_i=\ell_{t_{-\piv_i}}\in \ePet$.

\begin{lem}[{\cite[Lemma 2.30]{ISY}}]
\label{lem:W invariance of sigma}
For all $i\in I,$ $\sigma_i$ is invariant under the star action of $W.$ 
\end{lem}

\subsection{$K$-theoretic Peterson isomorphism }

For $w\in {W}$, let 
$\mathrm{Des}(w)
=\{i\in I\mid \OGr(ws_i)=\OGr(w)-1\},$
be the right descent set of $w$, 
and 
\begin{equation}
\gamma_w=-\sum_{j\in \Des(w)}\piv_j\in P^\lor.
\end{equation}
Then we have $wt_{\gamma_{x}} \in \hat{W}_{\af}$ (see \cite[Lemma 2.33]{ISY}).
For $w\in W$, define
the element
\begin{equation}
\bO^w:=
\frac{\OGr_{w t_{{\gamma_w}}}}{\OGr_{t_{{\gamma_w}}}}\in (\ePet)_{\loc}.
\end{equation}
\begin{prop} Let $w\in W$ and $\beta\in Q^\lor$ be anti-dominant so that 
$wt_\beta\in W_\af^0.$
Then 
\begin{equation*}
\frac{\OGr_{wt_\beta}}{\OGr_{t_\beta}}=\bO^w.
\end{equation*}
In particular, we have $\bO^w\in (\Pet)_{\loc}$.
\end{prop}
\begin{proof}
By using Proposition \ref{prop:t factor}, we have
$
\OGr_{wt_\beta}
\OGr_{t_{\gamma_w}}
=\OGr_{wt_{\beta}t_{\gamma_w}}
=\OGr_{wt_{\gamma_w}t_{\beta}}
=\OGr_{wt_{{\gamma_w}}}{\OGr_{t_\beta}}.
$
Hence the proposition holds.
\end{proof}


We use the following identification due to Lam, Schilling, and Shimozono \cite{LSS}: 
\begin{equation*}
K_*^T(\Gr_G)\cong
\Pet,\quad 
K_*^{T}(\Gr_{G})_\loc\cong
(\Pet)_\loc.
\end{equation*}
\begin{thm}[Kato \cite{Kato}, see also \cite{CL}, \cite{LLMS}]\label{thm:Kato}
There is an injective  homomorphism of $R(T)$-algebra 
$\Phi: K_*^T(\Gr_G)_\loc\rightarrow QK_T(G/B)_{\mathrm{loc}}$ such that
\begin{align}
\OGr_{t_\beta}&\mapsto Q^{\beta}\quad \text{for anti-dominant $\beta\in Q^\lor$,}\label{eq:Phi ell}
\\
\bO^w&\mapsto \CO^w\quad\text{for $w\in W$}.
\label{eq:Phi Q}
\end{align}
In particular, we have
\begin{equation}
\label{eq:Qk}
{\prod_{j\in I}\sigma_j^{-\langle \alpha_k^\lor,\alpha_j\rangle}}\mapsto Q_k\quad (k\in I).
\end{equation}
\end{thm}

Here we follow \cite[Conjecture 2]{LLMS}
for convention of the map $\Phi,$ which differs from that of \cite{Kato} due to a different choice of Borel subgroups in the definition of the Schubert varieties in 
\(G/B\) and $\Gr_G$.

The following fact generalizes 
a result by Chaput and Perrin \cite[Proposition 6.8]{CP} for the original (homology) Peterson isomorphism.

\begin{prop}\label{prop:W_commutes}
The star action of $W$ on $K_*^T(\Gr_G)\cong \Pet$ 
is compatible with the left action of $W$ on $QK_T(G/B),$ via the map $\Phi$ i.e.,
\begin{equation*}
    w^L\circ \Phi=\Phi\circ w\!*
    \quad\text{for $w\in W$ }. 
\end{equation*}
\end{prop}
Both $QK_T(G/B)$ and $ K_*^T(\Gr_G)$
have natural left $\mathbb{K}_\af$-module structures (\cite{Kato}, \cite{Orr}, \cite{ISY}), and the map $\Phi$
is a homomorphism of $\mathbb{K}_\af$-module (\cite[\S 2]{Kato}).
Although the above proposition can be deduced from this fact,
we present a direct computational proof, based on explicit formulas 
of the left action of $W$ in 
\S \ref{ssec:proof_W_commutes}.

%


\subsection{Seidel classes on the affine side}
\label{ssec:Seidel}

For all $i\in I$, there is $\kappa_i\in W_\af^0$ satisfying
\begin{equation}
t_{-\varpi_i^\lor}=\pi_i^{-1}\kappa_i.
\end{equation}
Therefore, by Lemma~\ref{lem:Sigma acts on L}, we have
\begin{equation}
\sigma_i=\OGr_{t_{-\piv_i}}
=\OGr_{\pi_i^{-1}\kappa_i}=
\OGr_{\pi_i^{-1}}
(u_{\pi_i^{-1}}*\OGr_{\kappa_i}).
\end{equation}


\begin{prop}
For any $i\in I^s$, we have
\begin{align}
\miw{i}t_{-\varpi_i^\lor}&=\pi_i^{-1},\label{eq:pi inverse}\\
\pi_i\miw{i}^{-1}\pi_i^{-1}&=\kappa_i\label{eq:min-i and kappa}.
\end{align} 
\end{prop}
\begin{proof}
 For any $i\in I$ (not necessarily special),
 the stabilizer of $\piv_i$ is the standard parabolic subgroup $W_{P_i}.$ 
It follows that
$$
w_\circ W_{P_i}=\{u\in W\mid u(\piv_i)=w_\circ (\piv_i)\}.
$$
Let $w_\circ^{P_i}$ denote 
the unique smallest element in the coset $w_\circ W_{P_i}$. For $i\in I^s$, $w_\circ^{P_i}$ coincide  with $\miw{i}$.
Now we assume $i\in I^s$. Then 
 \cite[\S 10.1 (8)]{LS:Acta} reads as \eqref{eq:pi inverse}.
 Equation \eqref{eq:min-i and kappa}
 is shown in \cite[Proposition 2.34]{ISY} ($\miw{i}$ is denoted by $u_i^{-1}$ in \cite{ISY}).
\end{proof}



For any special node $i$, we have
\begin{equation}
\bO^{\miw{i}}=\frac{\OGr_{\pi_i^{-1}}}{\sigma_i}
\label{eq:Ovi}
\end{equation}
from \eqref{eq:pi inverse}.

\begin{example}
In type $D_5$
, we have $I^s=\{1,4,5\}$. The group $\Sigma$ is cyclic of order $4$ and generated by $\pi_4.$ 

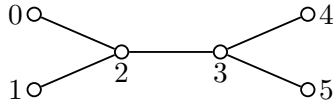
\begin{center}
\begin{tikzpicture}[
  scale=0.5,
  line width=0.7pt,
  dot/.style={circle,draw,inner sep=0pt,minimum size=5pt},
  lab/.style={font=\small, inner sep=0pt, outer sep=0pt} 
]
  \coordinate (C2) at (0,0);
  \coordinate (C3) at (2.6,0);
  \coordinate (L0) at (-2.3, 1.0);
  \coordinate (L1) at (-2.3,-1.0);
  \coordinate (R4) at ( 4.9, 1.0);
  \coordinate (R5) at ( 4.9,-1.0);

  \node[dot] (a2) at (C2) {};
  \node[dot] (a3) at (C3) {};
  \node[dot] (a0) at (L0) {};
  \node[dot] (a1) at (L1) {};
  \node[dot] (a4) at (R4) {};
  \node[dot] (a5) at (R5) {};

  \draw (a2)--(a3);
  \draw (a2)--(a0);
  \draw (a2)--(a1);
  \draw (a3)--(a4);
  \draw (a3)--(a5);

\node[lab,anchor=east]  at (-2.6,  1.00) {$0$};
\node[lab,anchor=east]  at (-2.6, -1.00) {$1$};
\node[lab,anchor=north] at ( 0.00, -0.3) {$2$};
\node[lab,anchor=north] at ( 2.60, -0.3) {$3$};
\node[lab,anchor=west]  at ( 5.2,  1.00) {$4$};
\node[lab,anchor=west]  at ( 5.2, -1.00) {$5$};
\end{tikzpicture}
%
%
%
%
%
\end{center}
We have $
\pi_4^2=\pi_1,\;
\pi_5=\pi_4^{-1}.
$
For $\pi_4$, the corresponding affine Dynkin automorphism is given by  
$0\mapsto 4,\;
1\mapsto 5,\;
2\mapsto 3,\;
3\mapsto 2,\;
4\mapsto 1,\;
5\mapsto 0.
$
We have
\begin{align*}
\kappa_4&=s_1s_2s_3s_5s_0s_2s_3s_1s_2s_0,\\
\miw{4}&=
s_5s_3s_4s_2s_3s_5s_1s_2s_3s_4.
\end{align*}
As a signed permutation,
$\miw{4}$ is $\bar{5}\bar{4}\bar{3}\bar{2}1$, sending $1$ to $\bar{5}$, $2$ to $\bar{4}$, etc.
We have 
$$
(\miw{4})^2=
\miw{1}=\bar{1}234\bar{5}.
$$
Some of the computations for type \(D_5\) in this paper were carried out
using the SageMath package for extended affine Weyl groups.
\end{example}



\subsection{Some properties of $\miw{i}$}
\label{sec:Seidel}




Let $R^{+}$ (resp. $ R^-$) be the set of all positive (resp. negative) roots.
For $w \in W$, set 
\begin{equation*}
    \inv(w) := \{\alpha \in R^{+} \mid w(\alpha) \in R^{-} \}. 
\end{equation*}
Note that 
\begin{equation}
    \Des(w) = \{ k \in I \mid \alpha_{k} \in \inv(w)\}. 
    \label{eq:Des roots}
\end{equation}
For $u, v \in W$, it is straightforward to see  
\begin{equation}
    \inv(vw) = (\inv(w) \setminus (-w^{-1}\inv(v))) \sqcup ((w^{-1}\inv(v)) \setminus \inv(w)).
    \label{eq:inv wv}
\end{equation}

\begin{lem}For $i \in I^{s}$, we have 
\begin{equation} \label{eq:long_rep_inv}
    \inv(\miw{i}) = \{ \alpha \in R \mid \pair{\varpi_{i}^{\vee}}{\alpha} = 1 \}.  
\end{equation}
\end{lem}
\begin{proof}
For any $i\in I$, we have
    \begin{equation}
\inv(\miw{i}) = \{ \alpha \in R^{+} \mid \pair{\varpi_{i}^{\vee}}{\alpha} > 0\}.
\label{eq:Inv min}
    \end{equation}
    Indeed, 
    $\miw{i}$ is the shortest element in $
w_{\circ}W_{P_i}$, therefore \eqref{eq:Inv min} holds by  \cite[(2.4.4)]{Mac}.
Then \eqref{eq:long_rep_inv} holds by \eqref{eq:cominuscule}. 
\end{proof}
\begin{remark}
For $i\in I^s$,
the set \eqref{eq:long_rep_inv} is known to equal to $R_{P_i}^{+},$ the positive part of the root system of $P_i.$
\end{remark}
\section{Proof of the Seidel product theorem }\label{sec:proof}

The following
is a key lemma for the proof of Theorem \ref{thm:main}. 
\begin{lem}\label{lem:key key} For $w\in W$ and $i\in I^s$, we have
\begin{equation}
\gamma_{\miw{i} w}=\gamma_w-w^{-1}(\piv_i).\label{eq:key}
\end{equation}
\end{lem}
\begin{proof}
By applying \eqref{eq:inv wv} with $w=\miw{i}$, we have
 \begin{align*}
 \mathrm{Inv}(\miw{i}w)
 &=\left(\Inv(w)\setminus
 (-w^{-1}\Inv(\miw{i}))\right)
 \sqcup
(w^{-1}\Inv(\miw{i})\setminus\Inv(w) ).
 \end{align*}
 By  \eqref{eq:long_rep_inv}
 we have
 $$
\pm w^{-1}\Inv(\miw{i})=\{
\alpha\in R\mid \pair{\piv_i}{w(\alpha)}=\pm 1\}.
$$
Therefore we have
\begin{align}
\Des(\miw{i}w)
&=I_1\sqcup I_2,\nonumber\\
I_1&=\{
j\in I\mid \alpha_j\in \Inv(w) \;\text{and}\;
\pair{\piv_i}{w(\alpha_j)}\ne -1\}\nonumber\\
&=\Des(w)
\setminus\{j\in I\mid 
\pair{\piv_i}{w(\alpha_j)}= -1\}\nonumber\\
&=\Des(w)\setminus\{j\in \Des(w)\mid 
\pair{\piv_i}{w(\alpha_j)}= -1\},\label{eq:Des minus}
\\
I_2&=
\{j\in I\mid \pair{\piv_i}{w(\alpha_j)}=1\;\text{and} \;\alpha_j\notin \Inv(w)\}\nonumber\\
&
\label{eq:I2 last}=\{j\in I\mid \pair{\piv_i}{w(\alpha_j)}=1\},
\end{align}
where \eqref{eq:Des minus}
 and \eqref{eq:I2 last} holds since
$\pair{\piv_i}{w(\alpha_j)}=1$ (resp. $-1$) implies that $w(\alpha_j)\in R^{+}$ (resp. $R^{-}$).
Therefore, we obtain 
\begin{align*}
\gamma_{\miw{i}w} 
    &=
    -\left(\sum_{j\in I_1}
    \piv_j+
    \sum_{j\in I_2}
    \piv_j\right)
\\&=    - \left( \left( \sum_{j \in \Des(w)} \piv_j - \sum_{\substack{j \in \Des(w) \\ \pair{\piv_{i}}{w(\alpha_{j})}=-1}} \piv_j \right) + \sum_{\substack{j \in I \\ \pair{\piv_i}{w(\alpha_{j})}=1}} \piv_j \right) \\&=    - \left( \left( \sum_{j \in \Des(w)} \piv_j - \sum_{\substack{j \in I \\ \pair{\piv_{i}}{w(\alpha_{j})}=-1}} \piv_j \right) + \sum_{\substack{j \in I \\ \pair{\piv_i}{w(\alpha_{j})}=1}} \piv_j \right)\\ 
\\&=    - \left( \left( \sum_{j \in \Des(w)} \varpi_{j}^{\vee} - \sum_{\substack{j \in I \\ \pair{w^{-1}(\piv_{i})}{\alpha_{j}}=-1}} \varpi_{j}^{\vee} \right) + \sum_{\substack{j \in I \\ \pair{w^{-1}(\piv_i)}{\alpha_{j}}=1}} \piv_j \right)\\
&=
\gamma_w
-\sum_{j\in I}
\pair{w^{-1}(\piv_i)}{\alpha_j}\piv_j\quad \text{(by \eqref{eq:cominuscule})} 
\\
    &= \gamma_{w} - w^{-1}(\piv_i), 
\end{align*}
where in the third equality we have used the fact that  
$\pair{\piv_i}{w(\alpha_j)}=-1$ implies $j\in\Des(w)$
again. \end{proof}

\begin{example}\label{ex:D5-2} In type $D_5$, 
we consider $w=s_{2}s_4s_3s_5s_3s_1s_2.$ Then,  $\Des(w)=\{2,5\}$ and $\Des(\miw{4}w)=\{1,3\}$. We have 
$w^{-1}(\piv_4)=
\piv_1-\piv_2+\piv_3-\piv_5.$
\end{example}

\begin{lem}
\label{lem:key}
Let $i$ be any special node, and $w\in W.$ We have
\begin{equation*}
\pi_i^{-1}wt_{\gamma_w}=
\miw{i} wt_{\gamma_{\miw{i} w}}.
\end{equation*}
\end{lem}
\begin{proof}From \eqref{eq:pi inverse}, we have
\begin{align*}
\pi_i^{-1}w
t_{\gamma_w}=
\miw{i} t_{-\piv_i}w
t_{\gamma_w}
=\miw{i}
wt_{-w^{-1}(\piv_i)}t_{\gamma_w}.
\end{align*}
Thus the lemma follows from Lemma \ref{lem:key key}.
\end{proof}

Let  
$x\in \Wexgr.$ Then 
 from \eqref{eq:pi on ell} and \eqref{eq:pi inverse} we have
\begin{equation}
\label{eq:miw i ell}
\miw{i}*\OGr_x=\OGr_{\pi_i^{-1}}^{-1}\OGr_{\pi_i^{-1}x}.
\end{equation}
 
\begin{proof}[Proof of Theorem \ref{thm:main}] We compute:
\begin{align*}
\bO^{\miw{i}}(
\miw{i}*\bO^{w})&=
\frac{\OGr_{\pi_i^{-1}}}{\OGr_{t_{-\piv_i}}}
\miw{i}*\left(
\frac{\OGr_{wt_{\gamma_w}}}{\OGr_{t_{\gamma_w}}}
\right)\quad \text{(by \eqref{eq:Ovi})}\\
&=\frac{\OGr_{\pi_i^{-1}}}{\OGr_{t_{-\piv_i}}}
\frac{(\OGr_{\pi_i^{-1}})^{-1}\OGr_{\pi_i^{-1}wt_{\gamma_w}}}{\OGr_{t_{\gamma_w}}}\quad\text{(by \eqref{eq:miw i ell} and Lemma \ref{lem:W invariance of sigma})}\\
&=\frac{1}{\OGr_{t_{-\piv_i}}}
\frac{\OGr_{\miw{i}wt_{\gamma_{\miw{i}w}}}}{\OGr_{t_{\gamma_w}}}\quad\text{(by Lemma \ref{lem:key})}\\
&=
\frac{\OGr_{t_{\gamma_{\miw{i}w}}}}{\OGr_{t_{-\piv_i}}\OGr_{t_{\gamma_w}}}
\bO^{\miw{i}w}\\
&=
\frac{\OGr_{t_{\gamma_w-w^{-1}(\piv_i)}}}{\OGr_{t_{-\piv_i}}\OGr_{t_{\gamma_w}}}
\bO^{\miw{i}w}\quad\text{(by Lemma \ref{lem:key key})}.
\end{align*}
Note that $\piv_i-w^{-1}(\piv_i)$ belongs to the nonnegative integer span of the simple coroots.
Then, by applying $K$-Peterson map, 
we obtain the theorem.
\end{proof}
\subsection{Proof of Corollary \ref{cor:P}}

Let $P$ be an arbitrary standard parabolic subgroup of $G$, and let
$\pi \colon G/B \to G/P$
be the natural projection. We recall a
pushforward result, due to Kato (\cite{Kato2}).

Denote by $W_P \subset W$ the Weyl group of $P$.
Recall that, for $w\in W$,  $\minrep{w}$ denote the unique minimal-length
representative of the coset $wW_P$. 
Let $Q^\lor_{\ge 0}:=\sum_{j\in I}\Z_{\ge 0}\alpha_j^\lor$.
\begin{prop}[Kato \cite{Kato2}, Theorem~2.19]\label{prop:push}
There exists a surjective homomorphism of $R(T)$-algebras
\begin{equation*}
\pi_* \colon
QK_T(G/B)
\longrightarrow
QK_T(G/P)
\end{equation*}
sending $\CO^w$ to $\COP^{\minrep{w}}$ for $w \in W$, and $Q^\beta$ to
$Q^{\minrep{\beta}}$ for 
$\beta\in Q^\lor_{\ge 0}.$
\end{prop}

We use the following standard lemma (see, e.g., \cite[Chapter~2]{BB}).
\begin{lem}\label{lem:std}
Let $w \in W^P$ and $i \in I$. If $s_i w > w$ and
$s_i w \notin W^P$, then there exists $j \in I_P$
such that $s_i w = w s_j$.
\end{lem}

The following result should be known to experts; we include a proof for completeness.

\begin{prop}\label{prop:pi W}
The pushforward 
$\pi_*: QK_T(G/B)\rightarrow QK_T(G/P)$ commutes with the left $W$-action. 
\end{prop}
\begin{proof}
According to \cite[Proposition 8.3 (b)]{MNS}, 
the left action of $W$ on $QK_T(G/P)$ for general parabolic subgroup $P$ is given for $w\in W^P$ and $i\in I$, by 
\begin{equation}
s_i^L \COP^{w}=\begin{cases}
e^{\alpha_i}\COP^{w}+(1-e^{\alpha_i})\COP^{{s_iw}} & \text{if $s_i w<w,$}\\
\COP^{w}& \text{otherwise.}
\end{cases}
\label{eq:def star on CO}
\end{equation}
 Note that this holds for $P=B$ too.

Since the left action of $W$ is a $\mathbb{Z}[Q]$-linear ring automorphism of $QK_T(G/B)$ (\cite[\S 8.3]{MNS}), it suffices to prove 
\begin{equation}
\pi_*(s_i^L\CO^w)=s_i^L\COP^{\minrep{w}}
\quad\text{for all $w\in W$ and $i\in I.$}
\label{eq:pi_commutes_with_W}
\end{equation}
We note that the following holds for $w\in W$ and $i\in I$:
\begin{equation}
s_i\minrep{w}\in W^P
\Longleftrightarrow
s_i\minrep{w}
=\minrep{s_iw}.
\label{eq:minrep}
\end{equation}

{\bf Case 1.} First we consider the case when  $\minrep{s_i w}
=\minrep{w}$ holds. Then 
it is easy from the definition of $\pi_*$ and \eqref{eq:def star on CO} to see that \begin{equation}
    \pi_*(s_i^L\CO^w)=\COP^{\minrep{w}}.\label{eq:pi_O} 
 \end{equation}
 Thus, in order to show \eqref{eq:pi_commutes_with_W}, it suffices to show that 
\begin{equation}
s_i^L\COP^{\minrep{w}}=\COP^{\minrep{w}}.
\label{eq:si inv}\end{equation}
If $s_i\minrep{w}>\minrep{w},$
then \eqref{eq:si inv} holds from \eqref{eq:def star on CO}.
Now suppose $s_i\minrep{w}<\minrep{w}.$
 Then we have $s_i\minrep{w}\in W^P$ (well known), and so $s_i\minrep{w}
=\minrep{s_iw}$ by \eqref{eq:minrep}. Then it is easy to see 
\eqref{eq:si inv} holds.


{\bf Case 2.} Next, we consider the case when 
$\minrep{s_iw}\ne \minrep{w}.$
We claim  
\begin{equation}
\label{eq:si minrep}
s_i\minrep{w}\in W^P.
\end{equation}
Indeed, if 
$s_i\minrep{w}< \minrep{w}$, then 
\eqref{eq:si minrep} holds.
Now assume $s_i\minrep{w}> \minrep{w}$ and $s_i\minrep{w}\notin W^P$. Then we have 
$s_i\minrep{w}W_P=\minrep{w}W_P$ by Lemma \ref{lem:std}, which means $\minrep{s_iw}=\minrep{w}.$ This contradicts the assumption, so \eqref{eq:si minrep} holds.

Now assume $s_iw<w.$
Then it ais straightforward to
see \eqref{eq:pi_commutes_with_W} holds.

Finally, we consider the case $s_iw>w$.
We have $\minrep{s_iw}\ge \minrep{w}$.
We know that $s_i\minrep{w}=\minrep{s_iw}$ holds since  \eqref{eq:si minrep} holds. Therefore, we have $s_i\minrep{w}>\minrep{w}.$
Then we have
$s_i^L\CO^w=\CO^w$ and $s_i^L\COP^{\minrep{w}}=\COP^{\minrep{w}}$, and 
\eqref{eq:pi_commutes_with_W} is fulfilled.
\end{proof}

\begin{proof}[Proof of Corollary \ref{cor:P}]
The corollary is 
obtained from Theorem \ref{thm:main} by Proposition \ref{prop:push} and Proposition \ref{prop:pi W}.
\end{proof}
 \appendix
\section{Extended $K$-Peterson algebra}\label{sec:ePetA}
The purpose of this section is to give a brief explanation of the construction in \cite[\S 2]{ISY}. 
\subsection{Extended level-zero $K$-nil-Hecke algebra $\hat{\mathbb{K}}_\af$}
\label{sec:K_nil}
Consider an action of $\hat{W}_\af$ on $R(T^\ad)=\Z[Q]$ given 
by 
\begin{equation*}
(t_\gamma w)e^\beta=e^{w(\beta)}\quad (\gamma\in P^\lor,\;w\in W,\; \beta\in Q),
\end{equation*}
where $w(\beta)$ denotes the natural action of $w$ on $\beta$.
This is called the \emph{level zero action}.
In particular $s_0$ acts on $\Z[Q]$ by $s_{\theta}$.
The action extends to the field $\F:=\mathrm{Frac}\,R(T^\ad)$ of fractions of $R(T^\ad)$.
The twisted group ring 
$\F[\hat{W}_\af]$ is
$\F\otimes_\Z \Z[\hat{W}_\af]$ as a left $\F$-module, 
equipped with the product
$$
(f\otimes w)(g\otimes v)=fw(g)\otimes wv \quad(f,g\in \F,\;
w,v\in W).
$$
We simply denote $f\otimes w$ by $fw.$
Let $\theta\in R$ be the highest root.
For $i\in I$, 
the corresponding simple root is denoted by  $\alpha_i$, here we understood $\alpha_0$ to be $-\theta$. 
For $i\in I_\af$, define
\begin{equation}
D_i=(1-e^{\alpha_i})^{-1}({s_i-1})+1
\in \F[W_\af].
\end{equation}
The elements $D_i\;(i\in I_\af)$ satisfy the braid relation and
$
D_i^2=D_i.
$ For $w\in W_\af$, take an reduced expression  $w=s_{i_1}\cdots s_{i_r}$. Then $D_w=D_{i_1}\cdots D_{i_r}$ is defined.

For $\sigma w\in \hat{W}_\af\;(\sigma \in\Sigma, \, w\in W_\af)$, define
$D_{\sigma w}=\sigma D_w.$ The extended level zero $K$- nil-Hecke algebra 
$\hat{\mathbb{K}}_\af$
is the left $R(T^\ad)$-module generated by 
$D_x\;(x\in \hat{W}_\af).$ This is indeed a ring, but the ``scalars'' $R(T^\ad)$ 
are not central. 
Let $\ePet$ be the centralizer subalgebra of $R(T^\ad)$ in $\hat{\mathbb{K}}_\af.$
Note that $t_\gamma\;(\gamma\in P^\lor)$
is an element of $\ePet.$


\subsection{Schubert basis of $\ePet$}
There is a fundamental action of $\hat{\mathbb{K}}_\af$ on the Peterson algebra $\ePet.$
By using the star action, we can construct the Schubert basis of $\ePet.$

\begin{defn}\label{defn:l-basis}
For $x\in \hat{W}_\af^0,$ define
$\OGr_x:=D_x*1\in \ePet.$
\end{defn}

\begin{prop}[{\cite[Proposition 2.24]{ISY}}]\label{prop:D acts on ell}
Let $i\in I_\af$ and $x\in \hat{W}_\af^0$. Then we have 
\begin{equation*}
D_i*\ell_x=\begin{cases}
    \ell_{s_i x} & \text{if $\ell(s_ix)>\ell(x)$ and $s_i x\in \hat{W}_\af^0$,}\\
    \ell_x &\text{otherwise}.
\end{cases}
\end{equation*}
\end{prop}
\begin{thm}[{\cite[Theorem 2.15]{ISY}}]
We have 
$\ePet=\bigoplus_{x\in \hat{W}_\af^0} R(T^\ad)\ell_x.$
\end{thm}

\begin{remark}Non-extended verion of 
$K$-nil-Hecke algebra 
$\mathbb{K}_\af$, and that of the Peterson algebra $\Pet
$
are defined similarly.
$\Pet$ has an $R(T)$-basis $\{\ell_x\mid x\in W_\af^0\}$ defined similarly, by using only $D_x$ for $x\in W_\af^0.$
\end{remark}

\subsection{Star action of $W_\af$ on Schubert basis elements}
\begin{prop}
For $i\in I_\af$, we have 
\begin{equation}
s_i* a=e^{\alpha_i}a
+(1-e^{\alpha_i})D_i *a
\quad \text{for all $a\in \ePet$}.
\label{eq:left action def}
\end{equation}
In particular, for $x\in \hat{W}_\af^0$ and $i\in I_\af$, we have
\begin{equation}
s_i* \ell_x=\begin{cases}
 e^{\alpha_i}\ell_x
+(1-e^{\alpha_i})\ell_{s_i x}& \text{if $s_ix>x$ and $s_i x\in \hat{W}_\af^0$,} \\
\ell_x & \text{otherwise}.
\end{cases}
\label{eq:left on ell}
\end{equation}
\end{prop}
\begin{proof}
\eqref{eq:left action def} is the definition of $D_i*a$.
\eqref{eq:left on ell} is equivalent to
Proposition \ref{prop:D acts on ell}.
\end{proof}


\subsection{Proof of Proposition \ref{prop:W_commutes}}
\label{ssec:proof_W_commutes}

\begin{lem}[{\cite[Lemma 6.7]{CP}}]\label{lem:CP}
Let $x\in W_\af^0$.
Write 
$x=wt_\beta$ with $w\in W$ and $\beta\in Q^\lor$. 
For $i\in I$, we have 
\begin{equation*}
s_ix>x\;\;\text{and}\;\; s_ix\in W_\af^0
\Longleftrightarrow
s_iw<w.
\end{equation*}
\end{lem}

\begin{proof}[Proof of Proposition \ref{prop:W_commutes}]
We will show
\begin{equation}
\Phi(s_i *\bO^w)=s_i^L 
\CO^w\quad \text{for $w\in W$ and $i\in I$.}
\label{eq:Phi_commutes_W}
\end{equation}
Recall the formula of $W$-action on $\{\CO^w\mid w\in W\}:$
\begin{equation}
\label{eq:W acts CO}
s_i^L
\CO^w=\begin{cases}
e^{\alpha_i}\CO^w+(1-e^{\alpha_i})\CO^{s_iw} & \text{if $s_i w<w$,}\\
\CO^w & \text{if $s_iw>w$.}
\end{cases}
\end{equation}
Then \eqref{eq:Phi_commutes_W}
is clear from Lemma \ref{lem:CP}, Lemma \ref{lem:W invariance of sigma} and \eqref{eq:left on ell}.
\end{proof}

\section{Examples}\label{sec:exa}
This appendix collects several examples illustrating applications of the results developed in the main text.
\subsection{Type $A_{n-1}$}
In type $A_{n-1}$,
we have $I=\{1,\ldots,n-1\}$.
All elements in $I$ are special and $\Sigma\cong \Z/n\Z.$
For $i\in I^s$,
$\miw{i}$ is the longest  
$i$-Grassmannian permutation $(i+1,i+2,\ldots,n,1,2,\ldots, i)$ in one-line notation.

In type $A_2$, we have
$
\miw{1}=s_2s_1,\; \miw{2}=s_1s_2,
$
and 
\begin{align*}
\bO^{s_1}=\frac{\ell_{\pi_1^{-1}s_0}}{\sigma_1},\quad
\bO^{s_2}=\frac{\ell_{\pi_2^{-1}s_0}}{\sigma_2},\quad
\bO^{s_1s_2}=\frac{\ell_{\pi_2^{-1}}}{\sigma_2},\quad
\bO^{s_2s_1}=\frac{\ell_{\pi_1^{-1}}}{\sigma_1},\quad
\bO^{w_\circ}=\frac{\ell_{s_0}}{\sigma_1\sigma_2},
\end{align*}
and $
Q_1=\dfrac{\sigma_2}{\sigma_1^2},\;
Q_2=\dfrac{\sigma_1}{\sigma_2^2},
$
where by abuse of notation we 
simply denote  $\prod_{j\in I}\sigma_{j}^{-\pair{\alpha_k^\lor}{\alpha_j}}$by $Q_k.$
Let us consider $i=2.$ For example, we have by \eqref{eq:miw i ell},
 \begin{align*}
 \bO^{\miw{2}}\cdot
 (\miw{2}*{\bO^{s_1}})&
=\frac{\ell_{\pi_2^{-1}}}{\sigma_2}
\miw{2}*\left({\frac{\ell_{\pi_1^{-1}s_0}}{\sigma_1}}\right)=\frac{\ell_{\pi_2^{-1}}}{\sigma_2}
{\frac{(\ell_{\pi_2^{-1}})^{-1}\ell_{\pi_2^{-1}\pi_1^{-1}s_0}}{\sigma_1}}\\
&=\frac{\ell_{s_0}}{\sigma_1\sigma_2}
=\bO^{w_\circ},
\end{align*}

The remaining cases (except for $w=e$) are given as follows: 
\begin{align*}
\bO^{\miw{2}}\cdot
(\miw{2}*{\bO^{s_2}})&=Q_2\bO^{s_1},\quad
\bO^{\miw{2}}\cdot(\miw{2}*{\bO^{s_1s_2}})=
Q_2\bO^{s_2s_1},\\
\bO^{\miw{2}}\cdot(\miw{2}*{\bO^{s_2s_1}})&=Q_1Q_2,\quad
\bO^{\miw{2}}\cdot(\miw{2}*{\bO^{w_\circ}})=Q_1Q_2\bO^{s_2}.
\end{align*}

\subsection{Type $C_2$}
In type $C_n$, $I=\{1,2,\ldots,n\}$ and we have $I^s=\{n\},$ and $
\miw{n}=s_n
(s_{n-1}s_n)\cdots (s_2 \cdots s_{n-1}s_n)
(s_1s_2 \cdots s_{n-1}s_n).
$

In type $C_2$, we have 
$\miw{2}=s_2s_1s_2$. We have
\begin{align*}
\bO^{s_1}&=\frac{\ell_{s_2s_1s_0}}{\sigma_1},\quad
\bO^{s_2}=
\frac{\ell_{\pi_2^{-1}s_1s_0}}{\sigma_2},\quad
\bO^{s_1s_2}=
\frac{\ell_{\pi_2^{-1}s_0}}{\sigma_2},\quad
\bO^{s_2s_1}=
\frac{\ell_{s_1s_0}}{\sigma_1},\\
\bO^{s_1s_2s_1}&=
\frac{\ell_{s_0}}{\sigma_1},\quad
\bO^{\miw{2}}=\bO^{s_2s_1s_2}
=\frac{\ell_{\pi_2^{-1}}}{\sigma_2},\quad
\bO^{w_\circ}=
\frac{\ell_{\pi_2^{-1}s_2s_1s_0}}{\sigma_1\sigma_2},
\end{align*}
and 
$
Q_1=\dfrac{\sigma_2^2}{\sigma_1^2},\quad
Q_2=\dfrac{\sigma_1}{\sigma_2^2}.
$ 
The following table lists $\bO^{v_2}\cdot v_2*\bO^w$ for $w\in W$.

\begin{center}
\renewcommand{\arraystretch}{1.4}
\begin{tabular}{c|c|c|c|c|c|c}
$s_1$
& $s_2$
& $s_1s_2$
& $s_2s_1$
& $s_1s_2s_1$
& $s_2s_1s_2$
& $w_\circ$ \\ \hline
$\bO^{w_\circ}$
& $Q_2\,\bO^{s_2s_1}$
& $Q_2\,\bO^{s_1s_2s_1}$
& $Q_1Q_2\,\bO^{s_2}$
& $Q_1Q_2\,\bO^{s_1s_2}$
& $Q_1Q_2^2$
& $Q_1Q_2^2\,\bO^{s_1}$
\end{tabular}
\end{center}

\subsection{Type $D_5$}
Consider type $D_5$ and $i=4.$
Let $w\in W$ be the element in Example \ref{ex:D5-2}. We have
$
\piv_4-w^{-1}\piv_4=
\alpha_2^\lor+\alpha_3^\lor+\alpha_4^\lor+\alpha_5^\lor.$
Therefore
we have
$$
\CO^{\miw{4}}
\cdot \miw{4}^L\CO^{w}
=Q_2Q_3Q_4Q_5
\CO^{\miw{4} w}.
$$

\end{document}